\documentclass[preprint]{imsart}

\usepackage{natbib}
\usepackage{amsmath,amsthm,ascmac,amssymb}
\RequirePackage[colorlinks,citecolor=blue,urlcolor=blue]{hyperref}

\startlocaldefs
\numberwithin{equation}{section}
\newtheorem{thm}{Theorem}[section]
\newtheorem{cor}{Corollary}[section]
\newtheorem{lem}{Lemma}[section]
\newtheorem{prop}{Proposition}[section]
\newtheorem{ass}{Assumption}[section]
\theoremstyle{plain}
\theoremstyle{remark}
\newtheorem{rem}{Remark}[section]
\def\R{{\mathbb R}}
\def\N{{\mathbb N}}
\def\Z{{\mathbb Z}}
\def\F{\mathcal{F}}
\def\sgn{{\rm sign}}
\def\ve{\varepsilon}
\endlocaldefs

\begin{document}

\begin{frontmatter}

\title{Weak convergences of marked empirical processes in a Hilbert space and their applications}

\runtitle{Marked empirical processes}

\begin{aug}

\author{\fnms{Koji} \snm{Tsukuda}\thanksref{a}\ead[label=e1]{ctsukuda@g.ecc.u-tokyo.ac.jp}}
\and
\author{\fnms{Yoichi} \snm{Nishiyama}\thanksref{b}\ead[label=e2]{nishiyama@waseda.jp}}
\address[a]{Graduate School of Arts and Sciences, the University of Tokyo, 3-8-1 Komaba, Meguro-ku, Tokyo 153-8902, Japan. \printead{e1}}
\address[b]{Faculty of International Research and Education, Waseda University,
1-6-1 Nishi-Waseda, Shinjuku-ku, Tokyo 169-8050, Japan. \printead{e2}}

\runauthor{K. Tsukuda and  Y. Nishiyama}

\affiliation{The University of Tokyo and Waseda University}

\end{aug}

\begin{abstract}
In this paper, weak convergences of marked empirical processes in $L^2(\mathbb{R},\nu)$ and their applications to statistical goodness-of-fit tests are provided, where $L^2(\mathbb{R},\nu)$ is the set of equivalence classes of the square integrable functions on $\mathbb{R}$ with respect to a finite Borel measure $\nu$. 
The results obtained in our framework of weak convergences are, in the topological sense, weaker than those in the Skorokhod topology on a space of c\'adl\'ag functions or the uniform topology on a space of bounded functions, which have been well studied in previous works.
However, our results have the following merits: (1) avoiding conditions which do not suit for our purpose; (2) treating a weight function which makes us possible to propose an Anderson--Darling type test statistics for goodness-of-fit tests. 
Indeed, the applications presented in this paper are novel.
\end{abstract}


\end{frontmatter}

\section{Introduction and main results}\label{sec:1}
This paper deals with the weak convergence of a certain sequence of marked empirical processes in $L^2$ space. 
Let us begin with preparing a minimal set of notations to describe our main theorems and the scientific background around our results.  
Let $\nu$ be a finite Borel measure on $\R$, and $L^2(\R,\nu)$ the set of equivalence classes of the square integrable functions on $\R$ with respect to $\nu$.
As for $L^2(\R,\nu)$, a inner product $\langle\cdot,\cdot\rangle$ defined by $\langle f,g \rangle = \int_\R f(x)g(x) \nu(dx)$ for $f,g\in L^2(\R,\nu)$ and a norm $\| \cdot\|$ defined by $\| f \| = \langle f,f \rangle^{1/2}$ for $f\in L^2(\R,\nu)$ are equipped.
For an interval $A$, the function $1_A(\cdot)$ is defined by $1_A(x)=1 \ (x\in A), \ 0 \ (x\not\in A)$.

For every positive integer $n$, let us introduce a filtered probability space $(\Omega^n,\F^n,{\bf F}^n =\{\F_i^n\}_{i\geq 0},P^n)$.
Let $\{ X_i^n \}_{i\geq 0}$ be a real-valued ${\bf F}^n$-adapted sequence and $\{m_i^n\}_{i\geq 1}$ a real valued ${\bf F}^n$-adapted martingale difference sequence (thus for every $i$, $m^n_i$ is $\F^n_i$-measurable and $E^n[m^n_i|\F^n_{i-1}]=0$ almost surely).
In this paper, we show weak convergences in $L^2(\R,\nu)$ of an empirical process marked by the martingale difference sequence $\{m_i^n \}_{i\geq1}$
\[
x \leadsto Z_n(x) = \sum_{i=1}^n 1_{(-\infty,x]}(X_{i-1}^n) m_i^n 
\]
and its weighted process with weight function $x\mapsto w(x)(>0)$
\[
x \leadsto Z^w_n(x) = w(x) Z_n(x) = \sum_{i=1}^n w(x) 1_{(-\infty,x]}(X_{i-1}^n) m_i^n,
\]
to  Gaussian processes $G$ and $G^w$, respectively.
The limits are 
\[ x \leadsto G(x) = B(\Psi(x)), \]
and
\[ x\leadsto G^w (x)  = w(x) B(\Psi(x)), \]
where $x\leadsto B(x)$ is a standard Brownian motion and $\Psi$ is the limit (for the exact sense of the limit, see Assumptions \ref{AssN} and \ref{Ass2} below) of 
\[ \sum_{i=1}^n 1_{(-\infty,x]}(X_{i-1}) E^n [(m^n_i)^2|\F^n_{i-1}]. \]

First, we provide a sufficient condition to show the weak convergence of $Z_n$.

\begin{ass}\label{AssN}
(i) It holds that
\begin{equation} 
\sum_{i=1}^n 1_{(-\infty,x]}(X_{i-1}^n) E^n[(m_i^n)^2 |\F_{i-1}^n] \to^p \Psi(x) \quad (\forall x\in \R) 
\label{AN1}
\end{equation}
as $n\to\infty$, where $x\mapsto \Psi(x)$ is a continuous nondecreasing function on $\R$ satisfying $\Psi(x)\downarrow0$ as $x\to-\infty$ and $\Psi(x)\uparrow \Psi(\infty)<\infty$ as $x\to\infty$, and $\to^p$ denotes the convergence in probability.

(ii) There exists a constant $\delta>0$ such that
\[ \sum_{i=1}^n E^n[|m_i^n|^{2+\delta} |\F_{i-1}^n] \to^p 0 \]
as $n\to\infty$.

(iii) There exists a measurable function $\phi$ on $\R$ such that for every $n\in \N$ and $i=1,\ldots,n$ there exist some nonnegative constants $c_i^n$ such that 
$\sup_n \sum_{i=1}^n c_i^n < \infty$
and that
$E^n[(m_i^n)^2|\F_{i-1}^n] \leq c_i^n \phi(X_{i-1}^n)$ almost surely.

(iv) All $X_i^n$'s have the same distribution as $\zeta$ such that $E[\phi(\zeta)]<\infty$.
\end{ass}

The first goal of this paper is to show the following theorem which asserts the weak convergence of $Z_n$ under Assumption \ref{AssN}.
Its proof will be presented in Section \ref{sec:2}. 

\begin{thm}\label{mainN}
Under Assumption \ref{AssN}, $Z_n$ converges weakly to $G$ in $L^2(\R,\nu)$ as $n\to\infty$.
\end{thm}

\begin{rem}
An important point of Theorem\ref{mainN} is that we avoid the assumption (B) in Lemma 3.1 of \cite{RefKS} which makes a restriction on the transition density of a discrete time Markovian process and does not suit for our diffusion process model considered in Section \ref{sec:4}. 
Although \cite{RefE} gave a result for a non-Markovian process, he assumed a condition on the smoothness (the condition (D) in his Theorem 1) of the model which also does not fit in our purpose. 
However, notice that our result does not cover theirs because they considered the weak convergence under the uniform metric. 
\end{rem}

Next, we provide a sufficient condition to show the weak convergence of $Z^w_n$.
Obviously, if we set $w(\cdot)=1$, then $Z^w_n$ becomes $Z_n$.
However, Theorem \ref{mainN} was separately stated, because \eqref{assE} is stronger than \eqref{AN1}.

\begin{ass}\label{Ass2}
(i)  It holds that
\begin{equation}
E^n \left[ \left| \sum_{i=1}^n 1_{(-\infty,x]}(X_{i-1}^n) E^n[(m_i^n)^2 |\F_{i-1}^n] - \Psi(x) \right| \right] \to 0  \quad (\forall x\in \R) \label{assE}
\end{equation}
as $n\to\infty$, where $x\mapsto \Psi(x)$ is a continuous nondecreasing function on $\R$ satisfying $\Psi(x)\downarrow0$ as $x\to-\infty$, $\Psi(x)\uparrow \Psi(\infty)<\infty$ as $x\to\infty$.
Moreover, there exists a nondecreasing function $\Phi$ such that 
\begin{equation}
 I_n(x) \leq \Phi(x) \quad (\forall x\in \R) 
 \label{assE2}
\end{equation}
for all sufficiently large $n$, where
\[ I_n(\cdot) = \sum_{i=1}^n E^n \left[ 1_{(-\infty,\cdot]}(X_{i-1}^n) (m_i^n)^2  \right] . \]
Furthermore, it holds that
\[ \int_\R (\Psi(x)+\Phi(x)) (w(x))^2\nu(dx)  < \infty. \]

(ii) There exists a constant $\delta>0$ such that
\[
\sum_{i=1}^n E^n[1_{( -\infty, x ]}(X_{i-1}^n) |m_i^n|^{2+\delta}] \to 0 \quad (\forall x\in \R)  
\]
as $n\to\infty$, and there exists  a function $\Lambda$ such that 
\begin{equation}\label{assL1}
 \sum_{i=1}^n E^n[1_{( -\infty, x ]}(X_{i-1}^n) |m_i^n|^{2+\delta}] \leq \Lambda(x) \quad (\forall x\in \R) 
\end{equation}
for all sufficiently large $n$ and that 
\[ \int_\R \Lambda(x) (w(x))^{2+\delta} \nu(dx) < \infty .  \]
\end{ass}

The second goal of this paper is to show the following theorem which asserts the weak convergence of $Z_n^w$ under Assumption \ref{Ass2}.
Its proof will be presented in Section \ref{sec:3}. 
From the practical viewpoint, the case where $w=(\Psi)^{-1/2}$ is important since it corresponds to the standardization.

\begin{thm}\label{main}
Under Assumption \ref{Ass2}, $Z_n^w$ converges weakly to $G^w$ in $L^2(\R,\nu)$ as $n\to\infty$.
\end{thm}

\begin{rem}
As for \eqref{assE}, it follows from a well-known fact on the uniform integrability that if
\[ \sum_{i=1}^n 1_{(-\infty,x]}(X_{i-1}^n) E^n[(m_i^n)^2 |\F_{i-1}^n] \to \Psi(x)  \quad {\rm a.s.} \quad (\forall x\in \R), \]
then \eqref{assE} is equivalent to the uniform integrability of
\[ \sum_{i=1}^n 1_{(-\infty,x]}(X_{i-1}^n) E^n[(m_i^n)^2 |\F_{i-1}^n] \]
for every $x\in\R$, and also equivalent to 
$ I_n(x)  \to \Psi(x)$ for every $x\in \R$.
\end{rem}

\begin{rem}
As for \eqref{assE2}, if we assume Assumption \ref{AssN} (iii)(iv), then
\begin{eqnarray*}
I_n(x) &=&  E^n \left[ \sum_{i=1}^n 1_{(-\infty,x]}(X_{i-1}^n) E^n[(m_i^n)^2 |\F_{i-1}^n] \right]
\\
&\leq& E^n \left[ \sum_{i=1}^n 1_{(-\infty,x]}(X_{i-1}^n) c^n_i \phi(X^n_{i-1}) \right] \\
&=& E \left[ \sum_{i=1}^n 1_{(-\infty,x]}(\zeta) c^n_i \phi(\zeta) \right] \\
&\leq& \left( \sup_{n} \sum_{i=1}^n c^n_i \right) E \left[ 1_{(-\infty,x]}(\zeta) \phi(\zeta) \right],
\end{eqnarray*}
so we can take $\Phi(x)$ as the right-hand side of the above display (if the integrability condition holds).
\end{rem}

Based on Theorems \ref{mainN} and \ref{main}, which are of interest in their own right, we discuss goodness-of-fit tests for stationary ergodic processes.
Specifically, we propose a Cram\'er--von Mises type statistic based on discrete time observation to test a simple hypothesis for a diffusion process and an Anderson--Darling type statistic for a time series.
Goodness-of-fit tests have been extensively studied in the literature because they are useful to judge that a mathematical model is acceptable to describe sampled data.
We refer to \cite{RefGC} for a review on the goodness-of-fit tests, whose Section 5 is devoted to tests when dependence is present.
Among abundant works treating goodness-of-fit tests for stochastic process models, we are interested in an approach based on empirical processes marked by residuals developed by \cite{RefKS} and \cite{RefE}.
Our limit theorems (Theorems \ref{mainN} and \ref{main}) do {\em not} include Theorem 2.1 of \cite{RefKS} or Theorem 1 of \cite{RefE}, but our results contain the following merits which are important in our applications.
The assumptions of Theorem 2.1 of \cite{RefKS} or Theorem 1 of \cite{RefE} do not suit our diffusion process setting, on the other hand Theorem \ref{mainN} can be applied.
Moreover, although a weak convergence of an Anderson--Darling type test statistic cannot be directly derived from the weak convergence in the Skorokhod space or $\ell^\infty$ space which have been established in \cite{RefKS} and \cite{RefE}, Theorem \ref{main} enables us to consider the Anderson--Darling type test statistic.
As a result, the applications presented in this paper are novel.

\begin{rem}
Based on smoothing, \cite{RefN,RefN09} and \cite{RefMNN} proposed the Kolmogorov--Smirnov type goodness-of-fit tests for time series models and diffusion processes (based on discrete time observation), respectively.
What they treated is not $Z_n(x)$ but its smoothed version using the Kernel density estimation.
\end{rem}

\begin{rem}
The goodness-of-fit test for diffusion processes based on continuous time observation, which is not studied in this paper, is considered in several works.
See for example \cite{RefDK}, \cite{RefK2}, \cite{RefNN} and referenecs therein.
\end{rem}

\begin{rem}
In this paper, we only consider simple hypotheses.
\cite{RefKS} have considered not only simple hypothesis but also parametric composite hypothesis based on the idea of the martingale transformation \citep{RefK}.
Considering parametric composite hypothesis is a possible direction in future researches.
\end{rem}

\section{Proof of Theorem \ref{mainN}}\label{sec:2}
By Prohorov's tightness criterion for Hilbert space valued random sequences (see, e.g., Theorem 1.8.4 of \cite{RefVW}), it suffices to show the following two lemmas.

\begin{lem}\label{MarN}
Under Assumption \ref{AssN} (i)(ii), $\langle Z_n, h \rangle$ converges in distribution to $\langle G, h \rangle$ for any $h \in L^2(\R,\nu)$ as $n\to\infty$.
\end{lem}

\begin{lem}\label{TightN}
Under Assumption \ref{AssN} (iii)(iv), it holds that
\[
 \lim_{J\to\infty} \limsup_{n\to\infty} E^n\left[ \sum_{j=J}^\infty \left\langle Z_n, e_j \right\rangle^2 \right] = 0.  
\]
\end{lem}

\subsection{Proof of Lemma \ref{MarN}}
Since
\begin{eqnarray*}
\langle Z_n^w, h \rangle
&=& \int_\R \sum_{i=1}^n 1_{(-\infty,x]}(X_{i-1}^n) m^n_i h(x) \nu(dx)  \\
&=& \sum_{i=1}^n \left( \int_\R 1_{(-\infty,x]}(X_{i-1}^n)  h(x) \nu(dx) \right) m^n_i,
\end{eqnarray*}
we shall apply the martingale central limit theorem for the martingale difference sequence
\[ \left\{ \left( \int_\R 1_{(-\infty,\cdot]}(X_{i-1}^n) h(\cdot) \nu(dx) \right) m^n_i \right\}_{i=1}^n. \]

It is not difficult to prove that Assumption \ref{AssN} (i)  leads
\[ \sup_{x\in\R} \left|R_n(x)\right| \to^p 0, \]
where
\begin{equation}
R_n(\cdot) = \sum_{i=1}^n 1_{(-\infty,\cdot]}(X^n_{i-1}) E^n[|m^n_i|^2|\F^n_{i-1}] - \Psi(\cdot)  .  \label{defR}
\end{equation}
Hence
\begin{eqnarray*} 
\lefteqn{
\sum_{i=1}^n \left( \int_\R 1_{(-\infty,x]}(X_{i-1}^n) h(x) \nu(dx) \right)^2 E^n[|m^n_i|^2|\F^n_{i-1}] 
}
\\
&=& \sum_{i=1}^n \int_\R\int_\R 1_{(-\infty,x \wedge y]}(X_{i-1}^n) h(x) h(y) \nu(dx) \nu(dy) E^n[|m^n_i|^2|\F^n_{i-1}] \\
&=& \int_\R\int_\R \Psi(x\wedge y) h(x) h(y) \nu(dx) \nu(dy)  \\
&&+ \int_\R\int_\R  R_n(x\wedge y)  h(x) h(y) \nu(dx) \nu(dy)  \\
&\to^p& \int_\R\int_\R \Psi(x\wedge y)h(x)h(y) \nu(dx) \nu(dy) .
\end{eqnarray*}
On the other hand, it holds that
\[ E[\langle G, h \rangle^2] 
= \int_\R \int_\R \Psi(x \wedge y) h(x) h(y) \nu(dx) \nu(dy) . \]

What is left is to show the Lyapunov-type condition
\begin{equation}
\sum_{i=1}^n E^n \left[ \left| \int_\R 1_{(-\infty,x]}(X_{i-1}^n) h(x) \nu(dx) \right|^{2+\delta} |m^n_i|^{2+\delta} | \F^n_{i-1} \right] \to^p 0. \label{NLG}
\end{equation}
From
\begin{eqnarray*}
 \left| \int_\R 1_{(-\infty,x]}(X_{i-1}^n) h(x) \nu(dx) \right|^{2+\delta} 
&\leq& \left( \int_\R 1_{(-\infty,x]}(X_{i-1}^n)  \nu(dx) \right)^{\frac{2+\delta}{2}} \| h \|^{2+\delta} \\
&\leq& (\nu(\R))^{\frac{2+\delta}{2}} \| h \|^{2+\delta} ,
\end{eqnarray*}
the left-hand side of \eqref{NLG} is bounded above by
\[  (\nu(\R))^{\frac{2+\delta}{2}} \| h \|^{2+\delta} \sum_{i=1}^n E^n \left[ |m^n_i|^{2+\delta} | \F^n_{i-1} \right], \]
which converges to 0 in probability by Assumption \ref{AssN} (ii).

This completes the proof.
\qed

\subsection{Proof of Lemma \ref{TightN}}

For simplicity, let us denote
\[ \xi^n_i(x) = 1_{( -\infty, x ]} (X_{i-1}^n) m_i^n  \quad (x\in \R). \]

It follows from Assumption \ref{AssN} (iii)(iv) that
\begin{eqnarray*} 
E^n[\langle \xi^n_i, e_j \rangle^2 | \F^n_{i-1}] 
&=& \langle 1_{( -\infty, \cdot ]} (X_{i-1}^n) , e_j \rangle^2 E^n[ |m_i^n|^2 | \F^n_{i-1}]  \\
&\leq& \langle 1_{( -\infty, \cdot ]} (X_{i-1}^n) , e_j \rangle^2 c^n_i \phi(X^n_{i-1}) \\
&=& c^n_i  \left\langle 1_{( -\infty, \cdot ]} (X_{i-1}^n) \sqrt{\phi(X^n_{i-1})}, e_j \right\rangle^2,
\end{eqnarray*}
which yields that
\[
\sum_{i=1}^n E^n \left[ \langle \xi^n_i, e_j \rangle^2 \right] 
\leq \sum_{i=1}^n c^n_i E\left[ \langle \tilde{\xi}, e_j \rangle^2 \right] 
\leq E\left[ \left\langle \eta , e_j \right\rangle^2 \right],
\]
where $x\leadsto\tilde\xi(x)$ is a $L^2(\R,\nu)$-valued random variable which follows the same distribution as $1_{(-\infty,x]}(\zeta)\sqrt{\phi(\zeta)}$, and 
\[ \eta = \left(\sup_n\sum_{i=1}^n c^n_i \right)^{1/2} \tilde{\xi}. \]
It follows that
\begin{eqnarray*}
E^n\left[ \sum_{j=J}^\infty \left\langle \sum_{i=1}^n \xi^n_i, e_j \right\rangle^2 \right] 
&=& \sum_{j=J}^\infty E^n\left[  \left\langle \sum_{i=1}^n \xi^n_i, e_j \right\rangle^2 \right]  \\
&=& \sum_{j=J}^\infty \sum_{i=1}^n E^n\left[  \left\langle \xi^n_i, e_j \right\rangle^2 \right]  \\
&\leq& \sum_{j=J}^\infty E \left[  \left\langle \eta, e_j \right\rangle^2 \right]  \\
&=& E \left[ \sum_{j=J}^\infty \left\langle \eta, e_j \right\rangle^2 \right] .
\end{eqnarray*}
Since $E[\|\eta\|^2]<\infty$, the dominated convergence theorem yields that
\[ \lim_{J\to\infty} E \left[ \sum_{j=J}^\infty \left\langle \eta, e_j \right\rangle^2 \right] 
= E \left[ \lim_{J\to\infty}  \sum_{j=J}^\infty \left\langle \eta, e_j \right\rangle^2 \right] =0. \]
This completes the proof.
\qed

\section{Proof of Theorem \ref{main}}\label{sec:3}

By Prohorov's criterion, it suffices to show the following two Lemmas.

\begin{lem}\label{Marginal}
Under Assumption \ref{Ass2}, $\langle Z_n^w, h \rangle$ converges in distribution to $\langle G^w, h \rangle$ for any $h \in L^2(\R,\nu)$ as $n\to\infty$.
\end{lem}

\begin{lem}\label{Tight}
Under Assumption \ref{Ass2} (i), it holds that
\[
 \lim_{J\to\infty} \limsup_{n\to\infty} E^n\left[ \sum_{j=J}^\infty \left\langle Z^w_n, e_j \right\rangle^2 \right] = 0.  
\]
\end{lem}

In the proofs of these lemmas, let $n$ be sufficiently large such that \eqref{assE} and \eqref{assL1} hold.

\subsection{Proof of Lemma \ref{Marginal}}
Since it holds that
\begin{eqnarray*}
\langle Z_n^w, h \rangle
&=& \int_\R \sum_{i=1}^n w(x) 1_{(-\infty,x]}(X_{i-1}^n)  m^n_i h(x) \nu(dx)  \\
&=& \sum_{i=1}^n \left( \int_\R w(x) 1_{(-\infty,x]}(X_{i-1}^n)  h(x) \nu(dx) \right) m^n_i,
\end{eqnarray*}
we shall apply the martingale central limit theorem for the martingale difference sequence
\[ \left\{ \left( \int_\R w(x) 1_{(-\infty,x]}(X_{i-1}^n) h(x) \nu(dx) \right) m^n_i \right\}_{i=1}^n. \]

First we show that
\begin{equation}\label{L31pf}
\sum_{i=1}^n \left( \int_\R w(x) 1_{(-\infty,x]}(X_{i-1}^n)  h(x) \nu(dx) \right)^2 E^n[(m^n_i)^2|\F^n_{i-1}]
\end{equation}
converges in first mean to
\[ E[\langle G^w, h \rangle^2] 
= \int_\R \int_\R \Psi(x \wedge y) w(x) w(y) h(x) h(y) \nu(dx) \nu(dy) . \]
Note that \eqref{L31pf} equals
\begin{eqnarray*}
\lefteqn{\sum_{i=1}^n \left( \int_\R \int_\R w(x) w(y)  1_{(-\infty,x \wedge y]}(X_{i-1}^n) h(x) h(y) \nu(dx) \nu(dy) \right) E^n[(m^n_i)^2|\F^n_{i-1}]}
\\
&=& \int_\R \int_\R  \sum_{i=1}^n 1_{(-\infty,x \wedge y]}(X_{i-1}^n) E^n[(m^n_i)^2|\F^n_{i-1}] w(x) w(y) h(x) h(y) \nu(dx) \nu(dy).
\end{eqnarray*}
We use the dominated convergence theorem to see
\begin{equation}\label{L31tg1}
\int_\R \int_\R E^n \left[  \left| R_n(x \wedge y) \right| \right] w(x) w(y)  |h(x) h(y)| \nu(dx) \nu(dy) \to 0,
\end{equation}
where $R_n(\cdot)$ is defined in \eqref{defR}, because
\begin{eqnarray}
\lefteqn{
E^n \left[ \left| \int_\R \int_\R  R_n(x \wedge y)  w(x) w(y) h(x) h(y) \nu(dx) \nu(dy) \right| \right] } \nonumber \\
&\leq& \int_\R \int_\R E^n \left[  \left| R_n(x \wedge y) \right| \right] w(x) w(y)  |h(x) h(y)| \nu(dx) \nu(dy).
\end{eqnarray}
From \eqref{assE}, we have
\[ E^n \left[  \left| R_n(x \wedge y) \right| \right] w(x) w(y)  |h(x) h(y)|  \to 0 \]
for every $x$ and $y$.
Moreover, it holds that
\begin{eqnarray*}
\lefteqn{
E^n \left[  \left| R_n(x \wedge y) \right| \right]w(x) w(y) |h(x) h(y)| 
}\\
&\leq& \left( I_n(x \wedge y) + \Psi(x \wedge y) \right)
w(x) w(y) |h(x) h(y)|  \\
&\leq& \left( \Psi(x \wedge y) + \Phi(x\wedge y) \right) w(x) w(y) |h(x) h(y)| \\
&\leq& \left( \sqrt{\Psi(x)\Psi(y)} + \sqrt{\Phi(x)\Phi(y)} \right) w(x) w(y) |h(x) h(y)| \\
&=& \sqrt{\Psi(x)\Psi(y)} w(x) w(y) |h(x) h(y)| + \sqrt{\Phi(x)\Phi(y)} w(x) w(y) |h(x) h(y)|
\end{eqnarray*}
for every $x$ and $y$, where we have used 
\begin{eqnarray*}
E^n[|R_n(x \wedge y)|] 
&\leq&
E^n \left[ \sum_{i=1}^n 1_{(-\infty,x \wedge y]}(X_{i-1}^n) E^n [(m^n_i)^2|\F^n_{i-1}]  \right] + \Psi(x\wedge y) \\
&=& I_n(x \wedge y) +  \Psi (x\wedge y)
\end{eqnarray*}
and $\Psi(x\wedge y) \leq \sqrt{\Psi(x) \Psi(y)}$ and $\Phi(x\wedge y) \leq \sqrt{\Phi(x) \Phi(y)}$ which follow from the monotonicity of $\Psi$ and $\Phi$.
Furthermore, it follows from the Schwartz inequality that
\begin{eqnarray*}
\lefteqn{
\int_\R \int_\R \sqrt{\Psi(x)\Psi(y)} w(x) w(y) |h(x) h(y)| \nu(dx) \nu(dy)
}\\
&\leq& \left( \int_\R \int_\R \Psi(x)\Psi(y) (w(x))^2 (w(y))^2 \nu(dx) \nu(dy) \right)^{1/2} \\
&& \quad \left( \int_\R \int_\R  |h(x) h(y)|^2 \nu(dx) \nu(dy) \right)^{1/2} \\
&=& \left\{ \int_\R \Psi(x) (w(x))^2 \nu(dx) \right\} \| h \|^2 
< \infty
\end{eqnarray*}
and
\begin{eqnarray*}
\lefteqn{\int_\R \int_\R \sqrt{\Phi(x)\Phi(y)} w(x) w(y) |h(x) h(y)| \nu(dx) \nu(dy)} \\
&\leq& \left\{ \int_\R \Phi(x) (w(x))^2 \nu(dx) \right\} \| h \|^2 
< \infty.
\end{eqnarray*}
Therefore, the dominated convergence theorem implies \eqref{L31tg1}.

Next we see the Lyapunov-type condition, that is to say, the nonnegative valued random variable
\[
\sum_{i=1}^n E^n \left[ \left| \int_\R w(x) 1_{(-\infty,x]}(X_{i-1}^n) h(x) \nu(dx) \right|^{2+\delta} |m^n_i|^{2+\delta} | \F^n_{i-1} \right] 
\]
converges to 0 in probability.
Since it follows from the Schwartz inequality that
\begin{eqnarray*}
\lefteqn{
\left| \int_\R w(x) 1_{(-\infty,x]}(X_{i-1}^n) h(x) \nu(dx) \right|^{2+\delta}
}\\
&\leq& \left( \int_\R (w(x))^2 1_{(-\infty,x]}(X_{i-1}^n) \nu(dx) \right)^{\frac{2+\delta}{2}} \| h \|^{2+\delta} ,
\end{eqnarray*}
it suffices to see the convergence of
\[ \sum_{i=1}^n  E^n \left[ \left( \int_\R (w(x))^2 1_{(-\infty,x]}(X_{i-1}^n) \nu(dx) \right)^{\frac{2+\delta}{2}} |m^n_i|^{2+\delta} \right]  \]
to 0.
Moreover, this display can be evaluated by 
\begin{eqnarray*}
\lefteqn{
E^n \left[ \sum_{i=1}^n \left( \int_\R (w(x))^2 1_{(-\infty,x]}(X_{i-1}^n) \nu(dx) \right)^{\frac{2+\delta}{2}} |m^n_i|^{2+\delta} \right]
}\\
&\leq& (\nu(\R))^{\frac{\delta}{2}} E^n \left[ \sum_{i=1}^n \int_\R 1_{(-\infty,x]}(X_{i-1}^n) (w(x))^{2+\delta} \nu(dx) |m^n_i|^{2+\delta} \right] \\
&=& (\nu(\R))^{\frac{\delta}{2}} \int_\R \sum_{i=1}^n E^n \left[ 1_{(-\infty,x]}(X_{i-1}^n) |m^n_i|^{2+\delta} \right] (w(x))^{2+\delta} \nu(dx),
\end{eqnarray*}
so the dominated convergence theorem yields that the right-hand side converges to 0.
Indeed, as for the integrand, it holds for every $x$ that
\[ \sum_{i=1}^n E^n \left[ 1_{(-\infty,x]}(X_{i-1}^n) |m^n_i|^{2+\delta} \right] (w(x))^{2+\delta} \to 0 \]
and 
\[ \sum_{i=1}^n E^n \left[ 1_{(-\infty,x]}(X_{i-1}^n) |m^n_i|^{2+\delta} \right] (w(x))^{2+\delta} \leq \Lambda(x) (w(x))^{2+\delta} \]
whose right-hand side is $\nu$-integrable.

This completes the proof.
\qed

\subsection{Proof of Lemma \ref{Tight}}
In this subsection, let us denote
\[ \xi^n_i(x) = w(x) 1_{( -\infty, x ]} (X_{i-1}^n) m_i^n \quad (x\in \R) \]
for simplicity.
It holds that
\begin{eqnarray}
\lefteqn{
E^n\left[ \sum_{j=J}^\infty \left\langle \sum_{i=1}^n \xi_i^n, e_j \right\rangle^2 \right] \nonumber
}\\
&=& E^n\left[ \left\| \sum_{i=1}^n \xi_i^n \right\|^2  - \sum_{j=1}^J \left\langle \sum_{i=1}^n \xi_i^n, e_j \right\rangle^2 \right] \nonumber \\
&=& E^n \left[ \left\| \sum_{i=1}^n \xi_i^n \right\|^2 \right] - \sum_{j=1}^J E^n\left[ \left\langle \sum_{i=1}^n \xi_i^n, e_j \right\rangle^2 \right] . \label{L32g1}
\end{eqnarray}

As for the first term in the right-hand side of \eqref{L32g1}, since
\[ E^n[\langle \xi^n_i, \xi^n_j\rangle] = E^n[E^n[\langle \xi^n_i, \xi^n_j\rangle|\F^n_{j-1}]]=0 \]
for $i<j$, it holds that
\begin{eqnarray*}
 E^n \left[ \left\| \sum_{i=1}^n \xi_i^n \right\|^2 \right] 
&=& E^n \left[ \sum_{i=1}^n\left\|  \xi_i^n \right\|^2 \right] + 2 \sum_{i=1}^{n-1} \sum_{j=i+1}^n E^n \left[ \langle \xi^n_i, \xi^n_j\rangle \right] \\
&=& E^n \left[ \sum_{i=1}^n\left\|  \xi_i^n \right\|^2 \right].
\end{eqnarray*}
The dominated convergence theorem yields that 
\begin{eqnarray*}
\lefteqn{
\lim_{n\to\infty} E^n \left[ \sum_{i=1}^n\left\|  \xi_i^n \right\|^2 \right]
}\\
&=& \lim_{n\to\infty} \int_\R \sum_{i=1}^n  E^n [ 1_{(-\infty,x]}(X^n_{i-1}) |m^n_i|^2 ] (w(x))^2  \nu(dx) \\
&=& \lim_{n\to\infty} \int_\R I_n(x) (w(x))^2  \nu(dx) \\
&=& \int_\R \Psi(x) (w(x))^2 \nu(dx) \\
&=& \int_\R E\left[(B(\Psi(x)))^2\right] (w(x))^2 \nu(dx) \\
&=& E[\| w B \circ \Psi \|^2],
\end{eqnarray*}
where $w B \circ \Psi$ means $w(\cdot) B(\Psi(\cdot))$.
That is because for every $x\in\R$ we have
$I_n(x) (w(x))^2 \to \Psi(x) (w(x))^2$
and 
$I_n(x) (w(x))^2 \leq  \Phi(x) (w(x))^2$
whose right-hand side is $\nu$-integrable.

As for the second term in the right-hand side of \eqref{L32g1}, since $\{ \langle \xi^n_i, e_j \rangle \}_{i=1}^n$ is a martingale difference sequence, we have
\begin{eqnarray*}
\lefteqn{
E^n\left[ \left\langle \sum_{i=1}^n \xi_i^n, e_j \right\rangle^2 \right]
}\\
&=&   E^n\left[ \left( \sum_{i=1}^n \left\langle \xi_i^n, e_j \right\rangle \right)^2 \right]\\
&=& \sum_{i=1}^n  E^n\left[ \left\langle \xi_i^n, e_j \right\rangle^2 \right]\\
&=&\sum_{i=1}^n E^n\left[ \int_\R \int_\R w(x) w(y) 1_{(-\infty,x\wedge y]}(X^n_{i-1})  e_j(x) e_j(y) \nu(dx) \nu(dy) (m^n_i )^2 \right] \\ 
&=&\int_\R \int_\R  I_n(x \wedge y) w(x) w(y) e_j(x) e_j(y) \nu(dx) \nu(dy) .
\end{eqnarray*}
Hence
\begin{eqnarray*}
\lefteqn{
\sum_{j=1}^J E^n\left[ \left\langle \sum_{i=1}^n \xi_i^n, e_j \right\rangle^2 \right]
}\\
&=&\int_\R \int_\R \left( I_n(x \wedge y)  w(x) w(y) \sum_{j=1}^J e_j(x) e_j(y) \right) \nu(dx) \nu(dy).
\end{eqnarray*}
The dominated convergence theorem yields that
\begin{eqnarray}
\lefteqn{
\lim_{n\to\infty} \int_\R \int_\R \left( I_n(x \wedge y)  w(x) w(y) \sum_{j=1}^J e_j(x) e_j(y) \right) \nu(dx) \nu(dy) 
} \nonumber \\
&=& \int_\R \int_\R \lim_{n\to\infty}  \left( I_n(x \wedge y)  w(x) w(y) \sum_{j=1}^J e_j(x) e_j(y) \right) \nu(dx) \nu(dy) . \label{L32pl}
\end{eqnarray}
That is because, as for the integrand, it holds that
\begin{eqnarray*}
\lefteqn{
\left| I_n(x \wedge y)  w(x) w(y) \sum_{j=1}^J  e_j(x) e_j(y) \right|
}\\
&\leq&  \Phi(x\wedge y) w(x) w(y) \sum_{j=1}^J  |e_j(x) e_j(y)| \\
&\leq& \sqrt{\Phi(x)\Phi(y)} w(x) w(y) \sum_{j=1}^J |e_j(x) e_j(y)|
\end{eqnarray*}
for every $x$ and $y$, and
\begin{eqnarray*}
\lefteqn{\int_\R \int_\R \sqrt{\Phi(x)\Phi(y)} w(x) w(y)  \sum_{j=1}^J |e_j(x) e_j(y)| \nu(dx) \nu(dy)} \\
&=& \sum_{j=1}^J  \int_\R \int_\R \sqrt{\Phi(x)\Phi(y)} w(x) w(y)  |e_j(x) e_j(y)| \nu(dx) \nu(dy) \\
&\leq& \sum_{j=1}^J \left( \int_\R\int_\R |e_j(x) e_j(y)|^2 \nu(dx) \nu(dy) \right)^{1/2} \\
&& \quad \left( \int_\R \int_\R \Phi(x)\Phi(y) (w(x))^2 (w(y))^2 \nu(dx)\nu(dy) \right)^{1/2} \\
&=& J \int_\R \Phi(x) (w(x))^2 \nu(dx) < \infty.
\end{eqnarray*}
Moreover, \eqref{L32pl} equals
\begin{eqnarray*}
\lefteqn{\int_\R \int_\R \Psi(x\wedge y) w(x) w(y) \sum_{j=1}^J e_j(x) e_j(y) \nu(dx) \nu(dy)} \\
&=& \sum_{j=1}^J  \int_\R \int_\R (\Psi(x)\wedge \Psi(y)) w(x) w(y) e_j(x) e_j(y) \nu(dx) \nu(dy). 
\end{eqnarray*}
On the other hand, it holds that
\begin{eqnarray*}
\lefteqn{
\sum_{j=1}^J E\left[ \left\langle w B\circ\Psi , e_j  \right\rangle^2 \right]
}\\
&=& \sum_{j=1}^J E\left[ \left( \int_\R  B(\Psi(x)) w(x) e_j(x) \nu(dx) \right)^2 \right] \\
&=& \sum_{j=1}^J  \int_\R \int_\R E[B(\Psi(x)) B(\Psi(y))] w(x) w(y) e_j(x) e_j(y) \nu(dx) \nu(dy) \\
&=& \sum_{j=1}^J  \int_\R \int_\R (\Psi(x)\wedge \Psi(y)) w(x) w(y) e_j(x) e_j(y) \nu(dx) \nu(dy) .
\end{eqnarray*}

From what have been already proven, 
\[ \lim_{J\to\infty} \limsup_{n\to\infty} E^n\left[ \sum_{j=J}^\infty \left\langle \sum_{i=1}^n \xi_i^n, e_j \right\rangle^2 \right] \]
equals
\begin{equation} \label{s3ll}
E\left[ \left\| w B\circ\Psi \right\|^2 \right]  - \lim_{J\to\infty}  E\left[ \sum_{j=1}^J  \left\langle w B\circ\Psi , e_j  \right\rangle^2 \right].
\end{equation}
Finally, the dominated convergence theorem yields that \eqref{s3ll} equals
\[E\left[ \left\| w B\circ\Psi \right\|^2 \right] - E\left[\sum_{j=1}^\infty  \left\langle w B\circ\Psi, e_j \right\rangle^2 \right] 
= 0. \]
This completes the proof.
\qed

\section{Application 1: Cram\'er--von Mises type goodness-of-fit test for drift parameters in diffusion processes}\label{sec:4}

In this section, we show the application of Theorem \ref{mainN} to the goodness-of-fit test for a diffusion process model.

\subsection{Problem setting and test procedure}
We consider a strictly stationary ergodic stochastic process $\{X_t \}_{t\geq0}$ which is a solution to a one-dimensional stochastic differential equation (SDE)
\begin{equation}
X_t = X_0 + \int_0^t S(X_{s}) ds + \int_0^t \sigma(X_{s}) dW_s \quad (t \geq 0),  \label{SDE}
\end{equation}
where the random variable $X_0$ is an almost surely finite initial value, $S(\cdot)$ is a measurable function in interest, $\sigma(\cdot)$ is a known measurable function and $t\leadsto W_t$ is a standard Wiener process defined on a stochastic basis $(\Omega, \F, (\F_t)_{t\in[0,\infty)}, P)$.
Let us list up some assumptions on the functions $S(\cdot)$ and $\sigma(\cdot)$.

{\bf (A1)} There exists a constant $ C> 0 $ such that
\[
|S(x)-S(y)| \leq C|x-y|, \quad
|\sigma(x)-\sigma(y)| \leq C|x-y| \quad 
(\forall x,y \in \R). 
\]

{\bf (A2)}
The process $ (X_{t})_{t \in [0,\infty)} $ is a solution to the SDE \eqref{SDE} for $ (S,\sigma) $ and it is stationary and ergodic with the absolutely continuous invariant law $ \mu_{S,\sigma} $ (that is, $ t^{-1}\int_{0}^tg(X_{s})ds \to^p \int_{\R}g(x)\mu_{S,\sigma}(dx) $ as $ t \to \infty $ for every $ \mu_{S,\sigma} $-integrable function $ g $). 
We also assume that
\[ \int_{\R}|x|^3 \mu_{S,\sigma}(dx) < \infty .\]

\begin{rem}
The assumption {\bf (A1)} implies that there exists a constant $ C'>0 $ such that $ |S(x)| \leq C'(1+|x|) $ and $ |\sigma(x)| \leq C'(1+|x|) $.
\end{rem}

In our problem, from the continuous stochastic process \eqref{SDE}, $\{ X_{t^n_i} \}_{i=1}^n$ is observed at discrete time points $0=t^n_0<t^n_1<\cdots<t^n_n$ satisfying
\begin{equation}
t_n^n \to \infty, \quad n \Delta_n^2 \to 0 \label{HFD}
\end{equation}
as $n\to\infty$, where 
\[ \Delta_n = \max_{1\leq i\leq n}|t^n_i - t^n_{i-1}|.\]

\begin{rem}
We propose an asymptotically distribution free tests based on the sampling scheme \eqref{HFD}, namely, {\rm high frequency data}.
We should mention that there is a huge literature on discrete time approximations of statistical estimators for diffusion processes; see, for example, the Introduction of \cite{RefGHR} for a review including not only high frequency cases but also low frequency cases. 
In our context of goodness-of-fit test, however, it seems difficult to obtain asymptotically distribution free results based on low frequency data. 
Our result for this problem is related to the preceding work, \cite{RefMNN}, who considered some Kolmogorov--Smirnov type tests based on smoothing. 
The ideal assertion for the Kolmogorov--Smirnov type tests is still an open problem because it needs a weak convergence theorem in $ \ell^\infty(\R) $. 
\end{rem}

Under the setting above, the problem is to conduct a goodness-of-fit test of \eqref{SDE}, that is to say, we wish to test the null hypothesis {\it$H_0:S=S_0$ versus $H_1:S\neq S_0$ for a given $S_0$ with $\sigma$ being a known function}.  
Let us define the test statistic
\begin{equation}
\mathcal{D}_n = \int_{\R} \frac{| U_n (x; S_0)|^2}{\Psi_{S_0,\sigma}(\infty)}  \frac{\Psi_{S_0,\sigma} (dx)}{\Psi_{S_0,\sigma}(\infty)} , \label{Ndef}
\end{equation}
where 
\[ 
x\leadsto U_n(x; S) = \frac{1}{\sqrt{t^n_n}} \sum_{i=1}^n 1_{(-\infty,x]}(X_{t^n_{i-1}})  \left\{ X_{t^n_{i}} - X_{t^n_{i-1}} - S(X_{t^n_{i-1}})(t^n_i - t^n_{i-1}) \right\},  \]
and
\[ \Psi_{S, \sigma}(\cdot) = \int_{-\infty}^\cdot \sigma(z)^2 \mu_{S,\sigma}(dz) . \]
As it is shown in the next subsection, the asymptotic null distribution of $\mathcal{D}_n$ is
\begin{equation}
\int_0^1 |B(u)|^2 du. \label{Nlim}
\end{equation}

\subsection{Justification of proposed procedure}
Let us asymptotically justify our test procedure.
Let us denote
\begin{equation}
m^n_i = \frac{\sigma(X_{t_{i-1}^n}) (W_{t_{i}^n}-W_{t_{i-1}^n})}{\sqrt{t_n^n}} \quad (i=1,\ldots,n), 
\label{mnid}
\end{equation} 
\[ \tilde{m}^n_i = \frac{X_{t^n_{i}} - X_{t^n_{i-1}} - S(X_{t^n_{i-1}})(t^n_i - t^n_{i-1}) }{\sqrt{t^n_n}} \quad (i=1,\ldots,n). \]
Suppose that $H_0$ is true.
Then, as it will be seen in the proof of Proposition \ref{dlem}, the sequence $\{\tilde{m}^n_i\}_{i=1}^n$ is close to $\{ m^n_i \}_{i=1}^n$ which is a martingale difference sequence with respect to the filtration $\{\F_{i-1}\}_{i=1}^\infty$, and Theorem \ref{mainN} yields the weak convergence in $L^2(\R, \Psi_{S_0,\sigma})$ of
\[ x \leadsto \sum_{i=1}^n 1_{(-\infty,x]}(X_{t^n_{i-1}}) m_i^n, \]
which will be denoted by $M^b_n(\cdot)$.

\begin{prop}\label{dlem}
Let $ \nu $ be any finite Borel measure on $ \R $. 
Assume {\bf (A1)} and {\bf (A2)}. 
Then, $ U_n(\cdot;S) $ converges weakly in $ L^2(\R,\nu) $ to $ B\circ \Psi_{S,\sigma}(\cdot)$ as $n\to\infty$ with \eqref{HFD}, where $ B(\cdot) $ is a standard Brownian motion and 
\[ \Psi_{S,\sigma}(\cdot)=\int_{-\infty}^{\cdot}\sigma(z)^2\mu_{S,\sigma}(dz) .\] 
\end{prop}

\begin{proof}
Define 
\[
x \leadsto M_n^a(x)=
\frac{1}{\sqrt{t_n^n}} \sum_{i=1}^n 1_{(-\infty,x]}(X_{t_{i-1}^n}) \int_{t_{i-1}^n}^{t_i^n}\sigma(X_{s})dW_{s}, 
\]
and
\[
x \leadsto M_n^b(x) =
\sum_{i=1}^n 1_{(-\infty,x]}(X_{t_{i-1}^n})  m^n_i,
\]
where $\{m^n_i\}_{i=1}^{n}$ is defined in \eqref{mnid}.

From \eqref{HFD}, it is easy to see that $ |U_n(\cdot;S)-M_n^a(\cdot)| $ converges in probability under the uniform metric, and thus also under the $ L^2(\R,\nu) $-metric. 

Let us show that $ M_n^a(\cdot)-M_n^b(\cdot) $ converges weakly in $ L^2(\R,\nu) $ to zero (the degenerate random field) and that $ M_n^b(\cdot) $ converges to $ B \circ \Psi_{S,\sigma} (\cdot) $;
then the assertion of the lemma follows from Slutsky's lemma.
To show these two weak convergence claims, we shall apply Theorem \ref{mainN} for 
\[ x \leadsto \sum_{i=1}^n \xi^n_i(x) \]
with
\begin{equation}
\xi_{i}^{n}(x) = \frac{1}{\sqrt{t_n^n}} 1_{(-\infty,x]}(X_{t_{i-1}^n}) \int_{t_{i-1}^n}^{t_i^n} (\sigma(X_{s})-\sigma(X_{t_{i-1}^n}))dW_s \quad (i=1,\ldots,n)
\label{nl31}
\end{equation}
and 
\begin{equation}
\xi_{i}^{n}(x) = 1_{(-\infty,x]}(X_{t_{i-1}^n}) m^n_i ,  \quad (i=1,\ldots,n)
\label{nl32}
\end{equation}
respectively. 
The condition (i) in Assumption \ref{AssN} for \eqref{nl31} where the limit is zero is clear, while that for \eqref{nl32} can be proven by using Lemma \ref{techlem} (iii). 
The condition (ii) in Assumption \ref{AssN} is indeed satisfied. 
The conditions (iii) and (iv) in Assumption \ref{AssN} is immediate from the stationarity (as for \eqref{nl31}, use also the latter inequality of Lemma \ref{techlem} (i)). 
This completes the proof.
\end{proof}

The limit random variable satisfies that 
\[
\int_{\R}
\frac{|B{(\Psi_{S,\sigma}(x)})|^2}{\Psi_{S,\sigma}(\infty)}
\frac{\Psi_{S,\sigma}(dx)}{\Psi_{S,\sigma}(\infty)}
=^d
\int_{0}^{1} |B(u)|^2du, 
\]
where the notation $=^d$ means the distributions are the same. 
So, by using the continuous mapping theorem, we obtain the following corollary. 

\begin{cor}\label{dthm}
Suppose that {\bf (A1)} and {\bf (A2)} are satisfied for a given, specific $ S_{0} $ and a known $ \sigma $. 
If $ H_{0} $ is true, then $\mathcal{D}_n$ converges in distribution to \eqref{Nlim} as $n\to\infty$ with \eqref{HFD}.
\end{cor}

To close this subsection, let us mention the consistency of the test.
Let us write the alternative hypothesis in interest as
\begin{equation} \label{alt}
\int_{\R} \left|\int_{-\infty}^{x}\{ S_{0}(z)- S(z) \} \mu_{S,\sigma}(dz) \right|^2 \Psi_{S_{0},\sigma}(dx) >0.
\end{equation}
Hereafter, \eqref{alt} is assumed to be true.
Observe that 
\begin{eqnarray*}
\lefteqn{ \Psi_{S_0,\sigma}(\infty) \mathcal{D}_n^{1/2}  }\\
&=& \left( \int_{\R}|U_n(x;S_{0})|^2 \Psi_{S_{0},\sigma}(dx) \right)^{1/2} \nonumber
\\
&\geq& 
\sqrt{t_n^n} \left( \int_{\R} \left| H_n(x) \right|^2 \Psi_{S_{0},\sigma}(dx) \right)^{1/2} 
- \left( \int_{\R}|U_n(x;S)|^2 \Psi_{S_{0},\sigma}(dx) \right)^{1/2}, 
\end{eqnarray*}
where 
\[H_n(\cdot)=\frac{1}{t_n^n}\sum_{i=1}^n 1_{(-\infty, \cdot]}(X_{t_{i-1}^n})\{ S_0(X_{t_{i-1}^n})- S(X_{t_{i-1}^n}) \}
|t_i^n-t_{i-1}^n|.\]
By using Proposition \ref{dlem} and the continuous mapping theorem, the second term of the right-hand side is $ O_{P}(1) $. 
To prove that the probability that the first term is bounded by $ M $ tends to zero as $ n \to \infty $ for any $ M> 0 $, let us first see that 
\[
H_n(x)
\to^p \int_{-\infty}^{x}\{ S_{0}(z)- S(z) \} \mu_{S,\sigma}(dz)
\]
for every $x \in \R$ which follows from Lemma \ref{techlem} (iii) presented in the next subsection.
It is easy to show that this convergence holds uniformly in $ x $. 
Hence
\[
 \int_{\R} \left| H_n(x) \right|^2 \Psi_{S_{0},\sigma}(dx)  \to^p
\int_{\R}
\left|\int_{-\infty}^{x}\{ S_{0}(z)- S(z) \} \mu_{S,\sigma}(dz) \right|^2 \Psi_{S_{0},\sigma}(dx) >0.
\]
Therefore, it holds that $ P(\mathcal{D}_n > M ) = P(\Psi_{S_0,\sigma}(\infty) \mathcal{D}_n^{1/2} > \Psi_{S_0,\sigma}(\infty) M^{1/2} ) \to 1$ for any constant $ M> 0 $.

\subsection{A technical lemma}

In this subsection, we show the following lemma which has already been used. 

\begin{lem}\label{techlem}
Let $ X $ be a solution of the SDE (\ref{SDE}) 
with $ (S,\sigma) $ satisfying {\bf (A1)}. 
Let $ p $ be a positive integer, and assume $ \sup_{t \in [0,\infty)}E|X_t|^{p} < \infty $. 

(i) There exists a constant $ C_{p,S,\sigma}>0 $ depending only on $ p $, $ (S,\sigma) $ such that if $ |t_{i}-t_{i-1}| \leq 1 $ then 
\[
E\left[\left. \sup_{s \in [t_{i-1}^n,t_{i}^n]} |X_{s}-X_{t_{i-1}^n}|^p \right| \F_{t_{i-1}^n} \right] 
\leq C_{p,S,\sigma}|t_{i}^n-t_{i-1}^n|^{p/2} (1+|X_{t_{i-1}^n}|)^p, 
\]
\[
E\left[\left. \sup_{s \in [t_{i-1}^n,t_{i}^n]}|X_{s}|^p \right| \F_{t_{i-1}^n} \right] 
\leq C_{p,S,\sigma}(1+|X_{t_{i-1}^n}|)^p. 
\]

(ii) 
For given $ p $ Lipschitz continuous functions $ {\bf g}=(g_1,...,g_p) $, there exists a constant $ C_{p,{\bf g},S,\sigma}>0 $ depending also on $ (S,\sigma) $ such that if $ |t_{i}^n - t_{i-1}^n| \leq 1 $ then 
\begin{eqnarray*}
\lefteqn{
E\left[\sup_{s \in [t_{i-1}^n,t_{i}^n]} \left|\prod_{j=1}^pg_j(X_{s}) -\prod_{j=1}^pg_j(X_{t_{i-1}^n})\right| \F_{t_{i-1}^n} \right]
}\\
&\leq& C_{p,{\bf g},S,\sigma}|t_{i}^n-t_{i-1}^n|^{1/2} (1+|X_{t_{i-1}^n}|)^p. 
\end{eqnarray*}

(iii) 
Assume that $ X $ is ergodic with the absolutely 
continuous invariant distribution $ \mu $. 
Let $ x \in \R $ and 
$ p-1 $ Lipschitz continuous functions $ {\bf g}=(g_1,...,g_{p-1}) $ 
such that that $ \prod_{j=1}^{p-1}g_j $ is $ \mu $-integrable be given. 
If $ \Delta_n \to 0 $ then it holds that 
\[
\frac{1}{t_n^n}\sum_{i=1}^n
1_{(-\infty,x]}(X_{t_{i-1}^n})
\prod_{j=1}^{p-1}g_j(X_{t_{i-1}^n})
|t_{i}^n-t_{i-1}^n|
\to^p \int_{-\infty}^{x}
\prod_{j=1}^{p-1}g_j(z)
\mu(dz). 
\]
(This assertion is true also for $ p=1 $ if we read 
$ \prod_{j=1}^{1-1}g_{j} \equiv 1 $.) 
\end{lem}

\begin{proof}
The assertion (i) is well-known; see, for example, \cite{RefKe}.
The assertion (ii) can be proven by using (i). 
Let us show (iii). 
We write $ g(z)=\prod_{j=1}^{p-1}g_{j}(z) $. 
We may assume that all $ g_j $'s are nonnegative without loss of generality. 
(For the general case, notice that $ g $ is represented as the sum of some terms of the form $ a \prod_{j=1}^{p-1}\widetilde{g}_j $ where  $ \widetilde{g}_j= g_j\vee0 $ or $(-g_j)\vee0 $ and $ a=1 $ or $ -1 $.) 
For any $ \varepsilon > 0 $, choose two Lipschitz continuous functions $ l,u $ such that $ l \leq 1_{(-\infty,x]} \leq u $ and that $ \int_{\R}|u(z)-l(z)|g(z)\mu(dz) < \varepsilon $. 
Then it holds that 
\begin{eqnarray*}
\lefteqn{
\frac{1}{t_n^n}\sum_{i=1}^n
1_{(-\infty,x]}(X_{t_{i-1}^n})g(X_{t_{i-1}^n})
|t_{i}^n-t_{i-1}^n|
}
\\
&\leq&
\frac{1}{t_n^n}\sum_{i=1}^n
u(X_{t_{i-1}^n})g(X_{t_{i-1}^n})
|t_{i}^n-t_{i-1}^n|
\\
&=&
\frac{1}{t_n^n}
\int_{0}^{t_n^n}
u(X_{s})g(X_{s})ds
+O_P(\Delta_n^{1/2})
\\
&\to^p&
\int_{\R}
u(z)g(z)\mu(dz)
\\
&\leq&
\int_{-\infty}^{x}
g(z)\mu(dz)
+ \varepsilon. 
\end{eqnarray*}
By doing the same argument replacing $ u $ by $ l $ we finally get 
\[
\left|
\frac{1}{t_n^n}\sum_{i=1}^n
1_{(-\infty,x]}(X_{t_{i-1}^n})g(X_{t_{i-1}^n})
|t_{i}^n-t_{i-1}^n|
-\int_{-\infty}^{x}
g(z)\mu(dz)
\right|
\leq
\varepsilon + o_{P}(1). 
\]
Since the choice of $ \varepsilon $ is arbitrary, we have proven the assertion of (iii). 
This completes the proof.
\end{proof}

\section{Application 2: Anderson--Darling type goodness-of-fit test for nonlinear time series}\label{sec:5}

In this section, we show the application of Theorem \ref{main} to the goodness-of-fit test for a Markovian nonlinear time series model.

\subsection{Problem setting and test procedure}

We consider a strictly stationary ergodic stochastic process $\{X_i\}_{i=-\infty}^\infty$ given by
\begin{equation}
X_i = S(X_{i-1}) + \sigma(X_{i-1}) \ve_{i} \quad (i \in \Z),  \label{NLTS}
\end{equation}
where $S(\cdot)$ is a measurable function, $\sigma(\cdot)$ is a known measurable function satisfying $\inf_{x\in\R} \sigma(x) > 0$, and $\{ \ve_i \}_{i=-\infty}^\infty$ is an unobserved iid sequence of absolutely continuous random variables satisfying $P(\ve_1 \leq 0 ) = 1/2$ and $\ve_i$ is independent of $X_{i-1}$ for all $i\in\Z$.
In this section, no moment condition on $\ve_1$ is assumed.

Let us introduce the following assumption on $S(\cdot)$ and $\sigma(\cdot)$.

{\bf (B)}
The process $\{ X_i \}_{i=-\infty}^\infty$ is stationary and ergodic with the absolutely continuous invariant law $\mu_{S,\sigma}$, where the ergodicity is in the sense of the almost sure convergence, that is to say,
\[ \frac{1}{n}\sum_{i=1}^n g(X_i) \to \int_\R g(x) \mu_{S,\sigma}(dx) \quad {\rm a.s.} \]
for every $\mu_{S,\sigma}$-integrable function $g (\cdot)$.
Moreover, the distribution function $\Psi_{S,\sigma}$ of $\mu_{S,\sigma}$ satisfies
\[ \int_\R \frac{\mu_{S,\sigma}(dx)}{\sqrt{\Psi_{S,\sigma}(x)}}  < \infty. \]

In our problem, from the stochastic process \eqref{NLTS}, a time series $\{ X_i \}_{i=0}^n$ is observed.

Under the setting above, the problem is to conduct a goodness-of-fit test of \eqref{NLTS}, that is to say, we wish to test the null hypothesis {\it$H_0:S=S_0$ versus $H_1:S\neq S_0$ for a given $S_0$ with $\sigma$ being a known function}.  
Let us define the test statistic
\begin{equation}
\mathcal{T}_n = \int_{\R} \frac{1}{n \Psi_{S_0,\sigma}(x)}  \left( \sum_{i=1}^n \sgn(X_i - S_0(X_{i-1})) 1_{(-\infty,x]}(X_{i-1}) \right)^2 \mu_{S_0,\sigma} (dx) , \label{Tdef}
\end{equation}
where $\sgn(\cdot)=-1_{(-\infty, 0)}(\cdot) + 1_{(0,\infty)}(\cdot)$.
As it is shown in the next subsection, the asymptotic null distribution of $\mathcal{T}_n$ is
\begin{equation}
\int_0^1 \frac{|B(u)|^2}{u} du . \label{Tlim}
\end{equation}

\begin{rem}
Our statistic contains $\sgn(\cdot)$ along the lines of \cite{RefEr} and Section 7.3 of \cite{RefN}.
Of course, if the corresponding required condition on $\{\ve_i\}_{i=1}^\infty$ is satisfied, other functions mentioned in \cite{RefKS} can be used.
Some examples are $f(\cdot)=\cdot$, $f(\cdot)=1_{(0,\infty)}(\cdot)-(1-\alpha)$, and other bounded functions.
A merit of $f(\cdot)=\sgn(\cdot)$ is robustness against outliers.
\end{rem}

\begin{rem}
Our procedure can be regarded as an Anderson--Darling type statistic in the sense of
\[ E\left[ \int_0^1 \frac{|B(u)|^2}{u} du \right] = 1. \]
\end{rem}

\subsection{Justification of proposed procedure}
Let us asymptotically justify our test procedure by using Theorem \ref{main}.
Let
\[ x \leadsto \xi^n_i(x) = \frac{1}{\sqrt{\Psi_{S_0,\sigma}(x)}} 1_{(-\infty,x]}(X_{i-1}) m^n_i \quad (i=1,\ldots,n), \]
where
\[ m^n_i = \frac{\sgn(X_i- S_0(X_{i-1}))}{\sqrt{n}}  \quad (i=1,\ldots,n).\]
Suppose that $H_0$ is true.
Then $\{m^n_i\}_{i=1}^n$ is a martingale difference sequence with respect to the filtration $\{\F_{i}\}_{i=0}^n$ where $\F_{i}=\sigma\{ X_j: 0 \leq j \leq i \}$ for $i=1,\ldots,n$, and it holds that
\[ (m^n_i)^2 = \frac{1}{n} \quad {\rm a.s.} \quad (i=1,\ldots,n). \]
We will use Theorem \ref{main} with $w=(\Psi_{S_0,\sigma})^{-1/2}$.
From the stationarity and ergodicity, Assumption \ref{Ass2} can be checked.
Indeed, it holds that
\[ E^n \left[ \frac{1}{n} \sum_{i=1}^n 1_{(-\infty,x]}(X_{i-1}^n) \right] 
= E \left[ \frac{1}{n} \sum_{i=1}^n 1_{(-\infty,x]}(\zeta) \right] 
= \Psi_{S_0, \sigma} (x) \]
where $\zeta$ is a random variable following $\mu_{S_0,\sigma}$, so if {\bf (B)} is satisfied then we are able to check Assumption \ref{Ass2} by taking $\Psi=\Phi=\Lambda=\Psi_{S_0,\sigma}$ and $\delta=1$.
We thus have
\[ \sum_{i=1}^n \xi^n_i \Rightarrow \frac{ B\circ \Psi_{S_0,\sigma} }{\sqrt{\Psi_{S_0,\sigma}}} \quad {\rm in} \quad L^2(\R,\nu) \]
as $n\to\infty$ for any finite Borel measure $\nu$.
Therefore, the continuous mapping theorem and 
\[ \int_\R \frac{|B(\Psi_{S_0,\sigma}(x))|^2}{\Psi_{S_0,\sigma}(x)} \mu_{S_0,\sigma}(dx)
= \int_\R \frac{|B(\Psi_{S_0,\sigma}(x))|^2}{\Psi_{S_0,\sigma}(x)} \frac{\mu_{S_0,\sigma}(dx)}{\Psi_{S_0, \sigma}(\infty)}
=^d  \int_0^1 \frac{|B(u)|^2}{u} du \]
yield the following assertion.

\begin{prop}
Suppose that {\bf (B)} is satisfied for a given, specific $ S_{0} $ and a known $ \sigma $. 
If $H_0$ is true, then $\mathcal{T}_n$ defined in \eqref{Tdef} converges in distribution to \eqref{Tlim} as $n\to\infty$.
\end{prop}

\begin{rem}
The weak convergence of 
\[ \sum_{i=1}^n \frac{1}{\sqrt{\hat\Psi_{n} (\cdot)}} 1_{(-\infty,\cdot]}(X_{i-1}) m^n_i \]
in $L^2(\R,\nu)$ is not demonstrated in this paper, where
\[ \hat\Psi_{n} (\cdot) = \frac{1}{n} \sum_{i=1}^n 1_{(-\infty,\cdot]}(X^n_{i-1}). \]
\end{rem}

Finally, we briefly discuss the consistency of the test.
Let us write the hypothesis in interest as
\begin{equation}
P(X_i - S_0(X_{i-1}) \leq 0 | \F_{i-1}) = \frac{1}{2} - \delta \quad {\rm a.s.} \quad (i=1,\ldots,n).
\end{equation}
Then the null hypothesis is $\delta=0$ and the alternative hypothesis is $0<|\delta|<1/2$.
Hereafter, $0<|\delta|<1/2$ is assumed.
From $\Psi_{S_0,\sigma} (x)\leq \Psi_{S_0,\sigma} (\infty)=1$, it follows that
\[ 
 \mathcal{T}_n ^{1/2}
\geq 
\left\{
\int_{\R} \frac{1}{n}  \left( \sum_{i=1}^n \sgn(X_i- S_0(X_{i-1})) 1_{(-\infty,x]}(X_{i-1}) \right)^2 \mu_{S_0 ,\sigma}  (dx)  \right\}^{1/2}. 
\]
The right-hand side of the above display is bounded below by
\begin{eqnarray*}
\lefteqn{
\sqrt{n} \times \left\{ 4\delta^2 \int_{\R} \left( \frac{1}{n} \sum_{i=1}^n 1_{(-\infty,x]}(X_{i-1}) \right)^2 \mu_{S_0,\sigma}  (dx) \right\}^{1/2} }\\
&& - \left\{ \int_{\R}  \left(  \sum_{i=1}^n  1_{(-\infty,x]}(X_{i-1}) \check{m}^n_i \right)^2 \mu_{S_0,\sigma}  (dx) \right\}^{1/2},
\end{eqnarray*}
where 
\[\check{m}^n_i= \frac{1}{\sqrt{n}} \left( \sgn(X_i-S_0(X_{i-1})) - 2\delta \right) \quad (i=1,\ldots,n) .  \]
The first term tends to positive infinity in probability since\[ \int_{\R} \left( \frac{1}{n} \sum_{i=1}^n 1_{(-\infty,x]}(X_{i-1}) \right)^2 \mu_{S_0,\sigma}  (dx) \to^p \int_\R (\Psi_{S,\sigma}(x))^2 \mu_{S_0,\sigma} (dx) \]
which follows from the ergodicity, whereas the second term is $O_P(1)$ which is a consequence of Theorem \ref{main} since 
$\left\{ \check{m}^n_i \right\}_{i=1}^\infty$
is a martingale difference sequence with respect to the filtration $\{\F_{i}\}_{i=0}^\infty$.
Therefore, it holds that $ P(\mathcal{T}_n > M ) =P(\mathcal{T}_n^{1/2} > M^{1/2} ) \to 1 $ for any constant $ M> 0 $.

\section*{Acknowledgements}
This work was partly supported by Japan Society for the Promotion of Science KAKENHI Grant Number 16H02791(KT), 18K13454(KT), 15K00062(YN) and 18K11203(YN).

\end{document}